\documentclass[11pt]{article}
\usepackage{amsmath,amsthm,amsfonts,amssymb,bm,wasysym}
\usepackage{epsfig}
\usepackage[usenames]{color}
\usepackage{verbatim}
\usepackage{hyperref}
\usepackage{multicol}
\usepackage{comment}
\usepackage{float}
\usepackage{graphicx}
\usepackage{centernot}

\usepackage[colorinlistoftodos]{todonotes}
\usepackage[normalem]{ulem}

\parskip \medskipamount
\newtheorem{theorem}{Theorem}[section]

\numberwithin{equation}{section}
\newtheorem{lemma}[theorem]{Lemma}
\newtheorem{proposition}[theorem]{Proposition}
\newtheorem{corollary}[theorem]{Corollary}

\newtheorem{remark}[theorem]{Remark}

\numberwithin{equation}{section}

\def\N{\mathbb{N}}
\def\Z{\mathbb{Z}}

\def\S{\mathcal{S}}

\def\EE{\mathcal{E}}

\renewcommand{\phi}{\varphi}
\renewcommand{\epsilon}{\varepsilon}

\def\calE{{\mathcal E}}

\allowdisplaybreaks

\newcommand{\1}{{\text{\Large $\mathfrak 1$}}}

\renewcommand{\emptyset}{\varnothing}

\newcommand{\til}{\widetilde}

\newcommand{\pr}[1]{\mathbb{P}\!\left(#1\right)}
\newcommand{\E}[1]{\mathbb{E}\!\left[#1\right]}
\newcommand{\estart}[2]{\mathbb{E}_{#2}\!\left[#1\right]}
\newcommand{\prstart}[2]{\mathbb{P}_{#2}\!\left(#1\right)}
\newcommand{\prcond}[3]{\mathbb{P}_{#3}\!\left(#1\;\middle\vert\;#2\right)}
\newcommand{\econd}[2]{\mathbb{E}\!\left[#1\;\middle\vert\;#2\right]}

\newcommand{\estarth}[2]{\mathbb{\widehat{E}}_{#2}\!\left[#1\right]}

\newcommand{\econdh}[2]{\mathbb{\widehat{E}}\!\left[#1\;\middle\vert\;#2\right]}

\newcommand{\tv}[2]{\left\|#1-#2\right\|_\mathrm{TV}}

\newcommand{\norm}[1]{\left\| #1 \right\|}

\newcommand{\tn}{|\kern-.1em|\kern-0.1em|}

\newcommand\be{\begin{equation}}
\newcommand\ee{\end{equation}}

\def\eps{\varepsilon}

\newcommand{\tmix}[2]{t_{\mathrm{mix}}(\epsilon, #1, #2)}
\newcommand{\tx}[1]{t_{\mathrm{mix}}(\epsilon, #1)}

\newcommand{\cps}[1]{\mathcal{P}_{\eta_0}\!\left(#1\right)}

\newcommand{\tb}[1]{\textbf{\color{blue}{#1}}}

\newcommand{\dntorus}{\mathbb{T}^{d,n}}

\begin{document}

\title{Quenched exit times for  random walk on dynamical percolation}

\author{Yuval Peres\thanks{Microsoft Research, Redmond WA, U.S.A.\ \ Email:
        \hbox{peres@microsoft.com}} \and Perla Sousi\thanks{University of Cambridge, Cambridge, UK.\ \ Email: \hbox{p.sousi@statslab.cam.ac.uk}} \and
        Jeffrey E. Steif\thanks{Chalmers University of Technology
and Gothenburg University, Gothenburg, Sweden.\ \ Email:
        \hbox{steif@chalmers.se}}
}

\maketitle
\thispagestyle{empty}

\begin{abstract}
We consider random walk on dynamical percolation on the discrete torus $\Z_n^d$. In previous work, mixing times of this process for $p<p_c(\Z^d)$ were obtained in the annealed
setting where one averages over the dynamical percolation environment. Here
we study exit times in the {\em quenched} setting, where we condition on a typical dynamical percolation environment. We obtain an upper bound for all $p$ which for $p<p_c$ matches the known lower bound. 
\medskip\noindent
\newline
\emph{Keywords and phrases.} Dynamical percolation, random walk, hitting times, mixing times.
 \newline
 MSC 2010 \emph{subject classifications.}
 Primary 60K35, 60K37
  \medskip\noindent
\end{abstract}

\section{Introduction}

In this paper, we study quenched mixing results for random walk on dynamical percolation on the torus~$\Z_n^d$ with parameters $p$ and $\mu\le 1/2$. Let each edge evolve 
independently where an edge in state 0 (absent, closed) switches to 
state 1 (present, open) at rate $p\mu$ and an edge in state 1 switches to 
state 0 at rate $(1-p)\mu$.  Let $(\eta_t)_{t\ge 0}$ denote the 
resulting Markov process on $\{0,1\}^{E(\Z_n^d)}$ whose stationary distribution is product measure
with density $p$, denoted by $\pi_p$; this model is called {\em dynamical percolation}. 
We next perform a  random walk on the evolving graph $(\eta_t)_{t\ge 0}$
by having the random walker at rate 1 choose a neighbour (in the original graph) uniformly
at random and move there if (and only if) the connecting edge is open at that time. Letting
$(X_t)_{t\ge 0}$ denote the position of the walker at time $t$, we have, when initial configurations
are given, that
\[
(M_t)_{t\ge 0}:= ((\eta_t,X_t))_{t\ge 0}
\]
is a Markov process while $(X_t)_{t\ge 0}$ of course is not. 

In \cite{PerStaufSteif}, a number of annealed results were obtained for this model where
one has $d$ and $p$ fixed while $\mu$ and $n$ are considered the important parameters 
with respect to which we want to study the model.
We summarise here the relevant results obtained in~\cite{PerStaufSteif} concerning mixing time.

Even though mixing times are traditionally defined only for Markov chains, one can easily adapt the definition to cases like $X$ above. For $\epsilon\in (0,1)$ and $\eta_0$ a configuration of edges we write 
\[
\tx{\eta_0} = \min\left\{t\geq 0: \, \max_{x} \norm{\prstart{X_t=\cdot}{x,\eta_0} - \pi}_{\rm{TV}} \leq \epsilon   \right\}.     
\]
We write $\prstart{}{\pi_p}$, when the environment process starts from stationarity. We then write
\[
\tx{\pi_p} = \min\left\{t\geq 0: \, \max_{x} \norm{\prstart{X_t=\cdot}{x,\pi_p} - \pi}_{\rm{TV}} \leq \epsilon   \right\}.  
\]

The following describes the subcritical picture very well.

\begin{theorem}[\cite{PerStaufSteif}] \label{thm:subOLD} 
For any $d\ge 1$, $\epsilon>0$ and $p\in (0,p_c(\mathbb{Z}^{d}))$, there exists 
$C=C(d,\eps,p)\in (0,\infty)$ and $n_0=n_0(d,\epsilon,p)\in \N$ such that, for all $n\geq n_0$ and for all $\mu\le 1/2$, we have 
\[
\frac{n^2}{C \mu} \le \tx{\pi_p} \le \sup_{\eta_0}\tx{\eta_0} \le \frac{Cn^2}{\mu}.
\]
\end{theorem}

The following yields lower bounds throughout the whole parameter space of $p$.

\begin{theorem}[\cite{PerStaufSteif}]\label{thm:supOLD}
{\rm{(i)}} Given $d\ge 1$ and $\eps >0$, there exist $C_1=C_1(d,\eps)>0$ and $n_0=n_0(d,\epsilon)$
such that, for all $p$, for all $n\ge n_0$ and for all $\mu\le 1/2$, we have
\[
t_{\rm{mix}}(\eps,\pi_p) \geq C_1 n^2.
\]
{\rm{(ii)}} Given $d\ge 1$, $p$ and $\epsilon< 1-\theta_d(p)$, there exists $C_2=C_2(d,p,\eps)>0$ 
and $n_0=n_0(d,p,\epsilon)$ 
such that, for all $n\ge n_0$ and for all $\mu\le 1/2$, we have
\begin{equation}\label{eq:SecondLowerBoundInSuper}
\tx{\pi_p}\geq \frac{C_2}{\mu}.
\end{equation}
In particular, for $\epsilon< 1-\theta_d(p)$, we get a lower bound for
$\tx{\pi_p}$ of order $n^2 + \frac{1}{\mu}$.
\end{theorem}

{\em Remarks}
(i). In the usual theory of Markov chains, a lower bound on the $\eps$-mixing time for a fixed~$\eps$ would yield
a bound of a similar order on the  mixing time when $\epsilon=1/4$; this is however not the case here 
which is not a contradiction
since $(X_t)_{t\ge 0}$ is not a Markov chain. \\
(ii). We believe (as stated in~\cite{PerStaufSteif}) that in the supercritical regime, 
the mixing time is much faster than in the subcritical regime and 
has order at most $\frac{1}{\mu}+ n^2$. Despite this, the methods in~\cite{PerStaufSteif} did not even 
yield the much larger (subcritical) upper bound of $\frac{n^2}{\mu}$ in the supercritical regime.
One of the corollaries of one of our main results is to obtain a related upper bound
uniform in $p$ (away from 0 and 1). 

The above results concerned {\sl annealed} mixing times, meaning that
the marginal distribution of $(X_t)_{t\ge 0}$ is studied. Here we study mixing and exit times of 
the conditional distribution of $(X_t)_{t\ge 0}$ given (typical) $(\eta_t)_{t\ge 0}$;
in other words, we study the {\em quenched} mixing and exit time behaviour of~$(X_t)_{t\ge 0}$.

For certain results, an annealed version immediately yields a quenched version. 
This is true (due to Fubini's Theorem) for almost sure results such as recurrence
and transience. Note, on the other hand, that annealed upper bounds on mixing times do not
automatically pass to quenched upper bounds on mixing times. One way to see this is to observe that 
if a convex combination of probability measures is close in total variation to some probability 
measure $\nu$, it still of course may be the case that all the probability measures appearing in 
the convex combination are far in total variation from $\nu$. An example of this, in the context 
of a Markov chain in a randomly evolving environment, is the following. Let $(M_n)_{n\ge 1}$ be 
an i.i.d.\ sequence of $2\times 2$ matrices where each matrix is either 
 \[ \left( \begin{array}{cc}
1 & 0 \\
0 & 1 \end{array} \right)\]
or
 \[ \left( \begin{array}{cc}
0 & 1 \\
1 & 0 \end{array} \right)\]
each with probability $1/2$. Let $(X_k)_{k\ge 0}$ 
be the process on $\{0,1\}$ which at time $n$ jumps according to the matrix $M_{n+1}$. 
It is clear that the annealed mixing time is 1 since, independent of the starting
distribution for $X_0$,
the distribution of $X_1$ is uniform. However the 
quenched mixing time is always infinite, since if we condition on any ``environment''
$(M_n)_{n\ge 1}$, the resulting time inhomogeneous Markov chain is such that
for every $k$, $X_k$ is deterministic.     On the other hand, a quenched mixing upper bound easily yields an annealed mixing upper bound. 

First of all we write $\prstart{\cdot}{x,\eta}$ for the probability measure of the walk, when the environment process is conditioned to be $\eta = (\eta_t)_{t\geq 0}$ and the walk starts from $x$. We write $\mathcal{P}$ for the distribution of the environment which is dynamical percolation on the torus, a measure on c\`adl\`ag paths $[0,\infty) \mapsto\{0,1\}^{E(\Z^d_{n})}$. We write $\mathcal{P}_{\eta_0}$ to denote the measure $\mathcal{P}$ when the starting environment is $\eta_0$. Abusing notation we write $\prstart{\cdot}{x,\eta_0}$ to mean the law of the full system when the walk starts from $x$ and the initial configuration of the environment is $\eta_0$. To distinguish it from the quenched law, we always write $\eta_0$ in the subscript as opposed to $\eta$.

Now we discuss hitting time bounds in both the quenched and annealed settings.
Let $A\subseteq \Z_n^d$. We denote by $\tau_A$ the first hitting time of $A$ by $X$, i.e.
\[
\tau_A = \inf\{t\geq 0: X_t\in A\}.
\]
\begin{theorem}\label{thm:hittingtimes}
	For all $d\geq 1$ and $\delta>0$, there exists $C=C(d,\delta)<\infty$ so that for all $p\in [\delta,1]$, for all $n$, for all $\mu \le 1/2$ and for all $\eps$, random walk in dynamical percolation 
on $\Z_n^d$ with parameters~$p$ and $\mu$ satisfies for all $A\subseteq \Z_n^d$ with $|A|\geq n^d/2$
\begin{align*}
\max_{\eta_0}\,&\cps{\eta=(\eta_t)_{t\geq 0}: \, \max_x\estart{\tau_A}{x,\eta} \geq \frac{Cn^2\log(n/\epsilon)}{\mu}} \leq \epsilon \text{ and}\\
&\max_{x,\eta_0}\estart{\tau_A}{x,\eta_0}\leq \frac{Cn^2}{\mu}.
\end{align*}
\end{theorem}

For $\epsilon \in (0,1)$, $x\in\Z_n^d$ and a fixed environment $\eta=(\eta_t)_{t\geq 0}$ we write
\[
\tmix{x}{\eta} = \min\left\{ t\geq 0:\tv{\prstart{X_t=\cdot}{x,\eta}}{\pi} \leq \epsilon \right\}.
\]
We also write 
\[
\tx{\eta} = \max_x \,\tmix{x}{\eta}
\]
for the quenched $\epsilon$-mixing time.
We remark that $\tx{\eta}$ could also be infinite.  Using the obvious definitions,
the standard inequality $t_{\rm{mix}}(\epsilon)\le \log_2(1/\epsilon) t_{\rm{mix}}(1/4)$ does not necessarily 
hold for time-inhomogeneous Markov chains and therefore also not for quenched mixing times.
Therefore, in such situations, to describe the rate of convergence to stationarity, it
is more natural to give bounds on~$t_{\rm{mix}}(\epsilon, \eta)$ for all $\epsilon$ rather than just considering $\epsilon=1/4$.

\begin{theorem}\label{thm:QuenchedMixingTime}
For all $d\ge 1$ and $\delta>0$, there exists 
$C=C(d,\delta)<\infty$ so that for all $p\in [\delta,1]$, for all $n$, for all $\mu \le 1/2$ and for all $\eps$, random walk in dynamical percolation 
on $\Z_n^d$ with parameters~$p$ and $\mu$ satisfies for all $x\in \Z_n^d$
\begin{equation}\label{eq:cor:supMeanSquaredDisplacement}
\max_{\eta_0}\cps{\eta=(\eta_t)_{t\geq 0}: \,\tmix{x}{\eta}\geq  \frac{Cn^2\log(1/\epsilon)}{\mu^4}}\leq \eps.
\end{equation}
\end{theorem}

We next discuss quenched lower bounds on the mixing time. It is now important whether
$p$ belongs to the sub or supercritical regime for percolation. In \cite{PerStaufSteif}, it was proved that
$\frac{n^2}{\mu}$ is the correct order of the (annealed) mixing time in the subcritical regime and it was 
conjectured there that the mixing time in the supercritical regime is much faster. 

\begin{proposition}\label{thm:QuenchedMixingTimeLower}
For all $d\ge 1$, $p\in (0,p_c(\mathbb{Z}^{d}))$, $\eps>0$ and $M$, there exists 
$\beta=\beta(d,p,\eps,M)>0$ and $n_0=n_0(d,p,\epsilon,M)$ so that if
$(\eta_t)_{t\ge 0}$ is dynamical percolation started in stationarity, then for all $n\geq n_0$ we have 
\begin{equation}\label{eq:cor:QuenchedUpper}
\mathcal{P}\left(\eta=(\eta_t)_{t\geq 0}:\, t_{\rm{mix}}(1-\epsilon,\eta)\le \frac{\beta n^2}{\mu}\right)\le \frac{1}{M}.
\end{equation}
\end{proposition}

\begin{proof}[\bf Proof]
Fix $d\ge 1$, $p\in (0,p_c(\mathbb{Z}^{d}))$, $\eps$ and $M$. 
Let $\sigma:=\min\{\eps^2,\frac{1}{M^2}\}$. By Theorem~\ref{thm:subOLD} there exists
$\beta=\beta(d,p,\sigma)$ so that for all large $n$ and for all $\mu\le 1/2$,
\[
\tv{\prstart{X_{\frac{\beta n^2}{\mu}} =\cdot}{0,\pi_p}}{\pi}\ge 1-\sigma.
\]
 Since
$$
\prstart{X_{\frac{\beta n^2}{\mu}}=\cdot}{0,\pi_p}=
\int\prstart{X_{\frac{\beta n^2}{\mu}}=\cdot}{0,\eta} \cps{(\eta_t)_{t\ge 0}}\,d\pi_p(\eta_0),
$$
convexity of the total variation norm in the sense that
$$
\tv{\int \mu_\alpha d\rho(\alpha)}{\pi}\le \int \tv{\mu_\alpha}{\pi}d\rho(\alpha)
$$
yields that
\begin{equation}\label{eq:TVinequality}
\int\tv{\prstart{X_{\frac{\beta n^2}{\mu}}=\cdot}{0,\eta}}{\pi}\cps{(\eta_t)_{t\geq 0})}\, d\pi_p(\eta_0)
\ge 1-\sigma,
\end{equation}
where $\eta=(\eta_t)_{t\geq 0}$.
This now implies that
$$
\mathcal{P}\left(\eta=(\eta_t)_{t\geq 0}: \, \tv{\prstart{X_{\frac{\beta n^2}{\mu}}=\cdot }{0,\eta}}{\pi}\le 1-\sigma^{\frac{1}{2}}\right)
\le \sigma^{\frac{1}{2}}
$$
Since $\sigma:=\min\{\eps^2,\frac{1}{M^2}\}$, this gives the result.
\end{proof}

The following is a quenched lower bound in the context of Theorem \ref{thm:supOLD}. The first statement is proved as the previous result. The second one follows from the fact that with high probability the environment will be such that there will exist a vertex isolated throughout the interval $[0,\beta/\mu]$.

\begin{proposition}\label{thm:supquenched}
Given $d\ge 1$, $\epsilon\in (0,1)$, $p<1$ and $M$, there exist $\beta>0$ and $n_0>0$
such that, for all $n\ge n_0$ and for all $\mu\le 1/2$, if $(\eta_t)_{t\ge 0}$ is dynamical percolation started in stationarity, then
\begin{align*}
\mathcal{P}\left(\eta=(\eta_t)_{t\geq 0}: \, t_{\rm{mix}}(\epsilon,\eta)\le \beta n^2\right)&\le \frac{1}{M} \,\text{ and }\\
\mathcal{P}\left(\eta=(\eta_t)_{t\geq 0}:\, \tx{\eta}\le \frac{\beta}{\mu}\right)&\le \frac{1}{M}. 
\end{align*}
\end{proposition}

Mixing times in the supercritical case will be studied in~\cite{PerSoSt}. Some of the results in this paper will be used there.

%
%
%

\section{General setup}

Given a general finite state Markov chain $p(\cdot,\cdot)$ with state space $\Omega$ and a
stationary distribution $\pi$, we let 
\begin{equation}\label{eq:DefnQ}
Q_{p}(A,B)= Q_{p,\pi}(A,B):=\sum_{x\in A,y\in B} \pi(x)p(x,y) 
\end{equation}
for $A,B\subseteq \Omega$. Also, for $S\subseteq \Omega$,
we let 
$$\phi_p(S)=\phi_{p,\pi}(S):=\frac{Q_{p}(S,S^c)}{\nu(S)}.$$ 
Observe that
$$
\phi_p(S)=\prcond{X_1\not\in S}{X_0\in S}{}
$$
where $(X_k)_{k\in \Z}$ is the stationary Markov chain associated to $p(x,y)$ and $\pi$. We call
$\phi_{p}(S)$ the expansion of $S$ (relative to the Markov chain $p(x,y)$ and 
stationary distribution $\pi$). Note that~$p(x,y)$ may have more than one stationary distribution 
and so we need to make $\pi$ explicit.

Finally we recall standard notation. Let $\mu$ and $\nu$ be two probability measures on the same space~$\Omega$. We write
\[
\chi^2(\mu,\nu) = \sum_{x\in \Omega} \nu(y) \left(\frac{\mu(y)}{\nu(y)} -1\right)^2 = \sum_y \frac{\mu(y)^2}{\nu(y)} -1.
\]
By Cauchy-Schwartz we have
\[
2\tv{\mu}{\nu} \leq \chi(\mu,\nu).
\]

We will now consider the following general set up of a finite state Markov chain 
in a Markovian evolving environment.

Let $\EE$ be a state space for a discrete time homogeneous Markov chain $\eta$ with transition matrix~$R$. Moreover, for every $\zeta\in \EE$ let
$(p_\zeta(x,y))_{x,y\in \S}$ be a transition matrix on another state space $\S$. Assume $\pi$ is a probability distribution on $\S$ which is stationary for $p_\zeta$ for all $\zeta\in \EE$ and has full support. 

We now define an annealed discrete time Markov chain $(\eta,X)$ on $\EE \times \S$ evolving as follows: when in state $(\zeta,x)$, it jumps to the state $(\zeta',x')$ by first choosing $\zeta'$ at random according to $R(\zeta,\cdot)$ and then choosing $x'$ at random according to $p_{\zeta'}(x,\cdot)$.
In symbols if $Q$ denotes the annealed transition matrix we have 
\[
Q((\zeta,x),(\zeta',x')) = R(\zeta,\zeta') p_{\zeta'}(x,x').
\]

Given a realisation of the environment process, $\eta=(\eta_i)_{i\geq 0}$, the coordinate $X$ becomes a time inhomogeneous Markov chain with transition matrix $p_{\eta_i}$ at time $i-1$.

Observe that since $\pi p_\zeta=\pi$ for all $\zeta\in\calE$, it follows easily that 
$\max_x\tv{\prstart{X_k=\cdot}{x,\eta}}{\pi}$ is decreasing in $k$ for any fixed environment
$\eta$. Next, as defined in the introduction we let
\[
\tx{\eta}
:=\inf\left\{k:\,\max_x\tv{\prstart{X_k=\cdot}{x,\eta}}{\pi}\le\eps\right\}=
\max_{x\in \S}\tmix{\eta}{x}
\]
be the $\eps$-mixing time in the environment $\eta$. 

The following general theorem yields quenched upper bounds on the mixing time in our general set up 
of a Markov chain in a Markovian evolving environment. We let $\pi_{\star}:=\min_x\pi(x)$ below. For $\zeta\in \EE$ and $S\subseteq \S$ we let
$$\varphi(\zeta,S):=\estart{\varphi_{p_{\eta_1}}(S)}{\zeta}.$$ 
Notice that in the previous expression we average over the new environment $\eta_1$, i.e.\ we run the environment process for one step starting from $\zeta$ and use the transition matrix that it yields. 
For  $r\in [\pi_\star,\frac{1}{2}]$, let
$$
\varphi(r):=\inf\{\varphi(\zeta,S):\zeta\in\calE, \pi(S)\le r\}
$$
and $\varphi(r):=\varphi(\frac{1}{2})$ for $r\ge \frac{1}{2}$.
Clearly $\varphi(r)$ is weakly decreasing in $r$. It is crucial for our applications that in the above definitions, 
the minimization of $S$ occurs outside of the expectation rather than inside. If the minimum occurred on the inside, 
then $\varphi(r)$ would be much smaller and the following result would be much weaker.

\begin{theorem}\label{thm:QuenchedMixingTimeGeneral}
Consider a finite state Markov chain~$X$ in a Markovian evolving environment satisfying $p_{\zeta}(x,x)\geq \gamma$ for all $\zeta$ and all $x$ with $\gamma\in (0,1/2]$.  For all $\eps >0$ and $x\in \S$ if
$$
n\ge 1+ \frac{2(1-\gamma)^2}{\gamma^2}\int_{4\pi(x)}^{4/\eps} \frac{du}{u\varphi^2(u)}
$$
then for all $\zeta \in \EE$,
$$
\mathcal{P}_{\zeta}\left(\eta=(\eta_t)_{t\geq 0}: \, \chi(\prstart{X_n=\cdot}{x,\eta},\pi)\ge \eps^{1/4}\right)\le \eps^{1/4}.
$$
\end{theorem}

\begin{remark}\label{rem:variant}
\rm{
We note that the above theorem remains true in the following variant of the Markov chain described above. Suppose that at every step, the chain $X$ remains in place with probability $1/2$ and with probability $1/2$ it jumps according to the transition matrix given by the environment at this time. When $X$ stays in place (because of laziness), then the environment at the next step also stays in place, otherwise it moves according to its transition matrix. So the transition matrix of the environment depends on the extra randomness coming from whether $X$ made an actual jump or not. For the changes needed in the proof see Remark~\ref{rem:newproof}.
}	
\end{remark}

\section{Random walk on dynamical percolation}

In this section we prove Theorem~\ref{thm:QuenchedMixingTime} using the general result Theorem~\ref{thm:QuenchedMixingTimeGeneral} stated in the previous section. 
Before starting the proof we introduce some notation and prove some preliminary results. 

Let $\calE:= D_{[0,1]}(\{0,1\}^{E(\Z_n^d)})$ be the space of right 
continuous paths with left limits from $[0,1]$ into $\{0,1\}^{E(\Z_n^d)}$. 
Let $\overline{\eta}_k:= \eta_{[k-1,k]}$
and $Z_k:= X_{k}$. Then it is easy to see that 
$((\overline{\eta}_k,Z_k))_{k\ge 0}$ is a Markov chain in a Markovian evolving environment.
It is clear that $(\overline{\eta}_k)_{k\ge 0}$ 
is a Markov chain with state space $\calE$ and that for all $\zeta\in\calE$, the corresponding 
Markov chain $p_\zeta$ simply corresponds to doing random walk on $\Z_n^d$ for time $1$
during which time the bond configuration evolution is fixed to be $\zeta$.

\begin{lemma}\label{lem:Binomial}
	For all $\delta>0$, there exists $\sigma=\sigma(\delta)>0$ so that 
for all $d\ge 1$, for all $n$, for all $\mu\le 1/2$, for all $p\in [\delta,1]$,
for all $A\subseteq  E(\Z_n^d)$ and for all $\eta_0$,
$$
\cps{\left|a\in A:\eta_t(a)=1\mbox{ for all } t\in \left[1/2,1\right]\right| \ge |A|\sigma\mu}\ge \sigma\mu.
$$
\end{lemma}

\begin{proof}[\bf Proof]
The left hand size is minimised when $\eta_0\equiv 0$. In this case, the left hand side equals
$\pr{{\rm{Bin}}(|A|,(1-e^{-\mu/2})pe^{-\mu(1-p)/2})\ge |A|\sigma\mu}$ where ${\rm{Bin}}(m,q)$ denotes a Binomial random variable
with parameters $m$ and $q$. 
Since $\mu\le 1/2$ and $p\ge\delta$, it is clear that there exists 
a $\sigma$ satisfying the requirements.
\end{proof}

We now let $\partial_E(S)$ denote the {\em edge boundary} of $S$ which is the set of edges from $S$ to $S^c$. 
We simply write $\phi_{\eta_{[0,1]}}$ to denote $\phi_{p_{\eta_{[0,1]}},\pi}$, i.e.\ we run the environment process for time $1$ starting from $\eta_0$ and use ${p_{\eta_{[0,1]}}}$ for the transition probability of the random walk.

\begin{lemma}\label{lem:KeyIsoRequirement}
For all $d$, there exists $c_d>0$ so that for all $n$, for all $\mu\le 1/2$, for all $p$, for all $\beta$,
for all nonempty $S\subseteq \Z_n^d$ with $\pi(S)\le 1/2$,
for all $\eta_{[0,1]}$ satisfying
$$
\left|e\in \partial_E(S):\eta_t(e)=1\mbox{ for all } t\in \left[1/2,1\right]\right| \ge |\partial_E(S)|\beta,
$$
we have 
$$
\phi_{\eta_{[0,1]}}(S)\ge \frac{c_d\beta}{n(\pi(S))^{1/d}}.
$$
\end{lemma}

\begin{proof}[\bf Proof]
Let $S_{\rm{good}}:=\{s\in S:\exists \, e \mbox{ from $s$ to $S^c$ which is open during 
$[1/2,1]$} \}$ and
let $S_{\rm{bad}}:=S\backslash S_{\rm{good}}$. (Note that $S_{\rm{good}}$ is a subset of the 
{\em internal vertex boundary} of $S$.) Note that 
$$
|S_{\rm{good}}|\ge \frac{1}{2d}\times \left|e\in \partial_E(S):\eta_t(e)=1\mbox{ for all } t\in \left[1/2,1\right]\right|
$$
and hence $|S_{\rm{good}}|\ge \frac{|\partial_E(S)|\beta}{2d}$. Since $\pi(S)\le 1/2$, by the standard 
isoperimetric inequality on $\Z_n^d$, we have that $|\partial_E(S)|\ge c'_d |S|^{(d-1)/d}$ for some 
universal constant $c'_d$ only depending on $d$. It follows that 
\begin{equation}\label{eq:BoundOnSgood}
|S_{\rm{good}}|\ge \frac{c'_d}{2d} |S|^{(d-1)/d}\beta.
\end{equation}
Consider now
\[
\phi_{{\eta_{[0,1]}}}(S)=\prcond{X_{1}\not \in S}{X_0\in S}{\eta_{[0,1]}}.
\]
The subscript $\eta_{[0,1]}$ means that the environment is fixed to be this realisation.
The conditioning $X_0\in S$ gives probability $1/|S|$ to each point in $S$. Since the uniform distribution is stationary for all realisations of the environment by the definition of the random walk, we infer that \begin{equation}\label{eq:BoundOnMarginal}
\max_{y\in \Z_n^d}\prcond{X_{\frac{1}{2}}=y}{X_0\in S}{{\eta_{[0,1]}}}\le \frac{1}{|S|}.
\end{equation}
Now 
\begin{align}\label{eq:Product}
\begin{split}
\prcond{X_{1}\not \in S}{X_0\in S}{{\eta_{[0,1]}}}\ge & 
\prcond{X_{\frac{1}{2}}\in S_{\rm{good}}\cup S^c}{X_0\in S}{{\eta_{[0,1]}}}\\
&\times
\prcond{X_{1}\not \in S}{X_0\in S,X_{\frac{1}{2}}\in S_{\rm{good}}\cup S^c}{{\eta_{[0,1]}}}.
\end{split}
\end{align}
By~\eqref{eq:BoundOnMarginal}, the first factor on the right hand side 
is at least $1-\frac{|S_{\rm{bad}}|}{|S|}$. This equals 
$\frac{|S_{\rm{good}}|}{|S|}$, which, by~\eqref{eq:BoundOnSgood}, is at least
$\frac{c'_d}{2d} |S|^{-1/d}\beta=\frac{c'_d}{2d} (\pi(S))^{-1/d}n^{-1}\beta$.
For the second term, if $X_{\frac{1}{2}} \in S_{\rm{good}}$, we fix
 an arbitrary edge
$e$ from  $X_{\frac{1}{2}}$ to $S^c$ which is open during $[1/2,1]$.
The probability that the random walk attempts only one jump during
$[\frac{1}{2},1]$ and the attempted jump is along this edge is at least a constant
$\gamma=\gamma(d)>0$ only depending upon $d$. On the other hand, if $X_{\frac{1}{2}}\in S^c$, there is a fixed probability the walk does not move, which we can also take to be $\gamma(d)$.

This gives that the left hand side of~\eqref{eq:Product} is at least $\frac{\gamma(d)c'_d}{2d} (\pi(S))^{-1/d}n^{-1}\beta$. Letting 
$c(d):=\frac{\gamma(d)c'_d}{2d}$ yields the claim.
\end{proof}

\begin{proof}[\bf Proof of Theorem~\ref{thm:QuenchedMixingTime}]

	We will apply Theorem~\ref{thm:QuenchedMixingTimeGeneral} with $\EE$ being the space of right continuous paths with left limits on $[0,1]$ and $Z_k=X_{k}$ as defined earlier. Observe that 
$$
\tx{\eta}\leq \tx{(\overline{\eta}_k)_{k\ge 0})}.
$$
We now show that 
for all $d\ge 1$ and $\delta>0$, there exists 
$C_1=C_1(d,\delta)>0$ so that for all $p\in [\delta,1]$, for all $\mu \le 1/2$,
for all $n$ and for all $\eta_0\in\{0,1\}^{E(\Z_n^d)}$,
\begin{align}\label{eq:goal}
\cps{\phi_{{\overline{\eta}}_1}(S)\ge \frac{C_1\mu}{n(\pi(S))^{1/d}}}\ge C_1\mu.
\end{align}
Fix $d$ and $\delta$. Choose $\sigma(\delta)$ 
from Lemma~\ref{lem:Binomial} and $c_d$ from Lemma~\ref{lem:KeyIsoRequirement}.
Fix $S$ with $\pi(S)\le 1/2$. Combining Lemmas~\ref{lem:Binomial} and~\ref{lem:KeyIsoRequirement}
with $A$ in Lemma~\ref{lem:Binomial} taken to be $\partial_E(S)$
and~$\beta$ in~\ref{lem:KeyIsoRequirement} taken to be~$\sigma\mu$, the two lemmas imply that 
$$
\cps{\phi_{\eta_{[0,\frac{1}{\mu}]}}(S)\ge \frac{c_d\sigma \mu}{n(\pi(S))^{1/d}}}\ge \sigma\mu,
$$
establishing~\eqref{eq:goal}.
From~\eqref{eq:goal} we now get that for all sets $S$ with $\pi(S)\leq 1/2$
\begin{align*}
\phi(\eta_0,S) \geq  \frac{C_1^2\mu^2}{n(\pi(S))^{1/d}},
\end{align*}
and hence 
\begin{align*}
\int_{4\pi(x)}^{4/\epsilon} \frac{du}{u\phi^2(u)} = \int_{4/n^d}^{1/2} \frac{du}{u\phi^2(u)} + \frac{1}{\phi^2(1/2)} \int_{1/2}^{4/\epsilon} \frac{du}{u} \leq C_2\left(\frac{n}{\mu^2}\right)^2 \log\left(\frac{1}{\epsilon}\right),
\end{align*}
where $C_2$ is a positive constant. 
Since for any environment $\eta$ and any $x\in \Z_n^d$ we have
\[
\prcond{X_{1}=x}{X_{0}=x}{\eta} \geq \frac{1}{e},
\] 
applying Theorem~\ref{thm:QuenchedMixingTimeGeneral} completes the proof.
\end{proof}

We turn to prove Theorem~\ref{thm:hittingtimes}. We now let $\til{\calE}:= D_{[0,1/\mu]}(\{0,1\}^{E(\Z_n^d)})$ be the space of right 
continuous paths with left limits from $[0,1/\mu]$ into $\{0,1\}^{E(\Z_n^d)}$. 
Let $\overline{\eta}_k:= \eta_{[\frac{k-1}{\mu},\frac{k}{\mu}]}$
and $Z_k:= X_{\frac{k}{\mu}}$. Then again
$((\overline{\eta}_k,Z_k))_{k\ge 0}$ is a Markov chain in a Markovian evolving environment.
It is clear that $(\overline{\eta}_k)_{k\ge 0}$ 
is a Markov chain with state space $\til{\calE}$ and that for all $\zeta\in\til{\calE}$, the corresponding 
Markov chain $p_\zeta$ simply corresponds to doing random walk on $\Z_n^d$ for time $1/\mu$
during which time the bond configuration evolution is fixed to be $\zeta$.

The following lemma follows in exactly the same way as Lemma~\ref{lem:Binomial}.

\begin{lemma}\label{lem:Binomial1}
For all $\delta>0$, there exists $\sigma=\sigma(\delta)>0$ so that 
for all $d\ge 1$, for all $n$, for all $\mu\le 1/2$, for all $p\in [\delta,1]$,
for all $A\subseteq  E(\Z_n^d)$ and for all $\eta_0$,
$$
\cps{\left|a\in A:\eta_t(a)=1\mbox{ for all } t\in \left[\frac{1}{\mu}-1,\frac{1}{\mu}\right]\right| \ge |A|\sigma}\ge \sigma.
$$
\end{lemma}

\begin{proof}[\bf Proof]
As before, the left hand size is minimised when $\eta_0\equiv 0$. In this case, the left hand side equals
$\pr{{\rm{Bin}}(|A|,(1-e^{\mu-1})pe^{-\mu(1-p)})\ge |A|\sigma}$ where ${\rm{Bin}}(m,q)$ denotes a Binomial random variable
with parameters $m$ and $q$. 
Since $\mu\le 1/2$ and $p\ge\delta$, it is clear that there exists 
a $\sigma$ satisfying the requirements.
\end{proof}

\begin{proof}[\bf Proof of Theorem~\ref{thm:hittingtimes}]

In order to prove the theorem we first consider a lazy version of the Markov chain $((\overline{\eta}_k,Z_k))_{k\geq 0}$ as follows. At every step the walk remains in place with probability $1/2$ and with probability $1/2$ it jumps according to the transition matrix given by the environment at this time. When $X$ stays in place, then the environment at the next step also stays in place, otherwise it moves according to its transition matrix. This is the setup of Remark~\ref{rem:variant}.

For this new chain the statement of Lemma~\ref{lem:KeyIsoRequirement} remains the same with an extra factor of $1/2$ in the lower bound for $\phi_{\eta_{[0,1/\mu]}}$. Also Lemma~\ref{lem:Binomial1} still holds with an extra factor of $1/2$ again in the lower bound of the probability.

Remark~\ref{rem:variant} now shows that the statement of  Theorem~\ref{thm:QuenchedMixingTimeGeneral} remains true, and hence we  obtain 
\begin{align}\label{eq:mixqu}	
\max_{\eta_0}\cps{\eta=(\eta_t)_{t\geq 0}: \, \tmix{x}{\eta}\geq \frac{Cn^2 \log(1/\epsilon)}{\mu}}\leq \epsilon,
\end{align}
where the mixing time refers to the lazy version of the discretised random walk.	 By letting $\epsilon=\epsilon/n^d$ and taking a union bound over all $x\in \Z_n^d$  we obtain
\[
\max_{\eta_0}\cps{\eta=(\eta_t)_{t\geq 0}: \, \tx{\eta}\geq \frac{Cn^2 \log(n/\epsilon)}{\mu}}\leq \epsilon.
\]
Using this, we can obtain an upper bound of the same order for the hitting time of any set $A\subseteq \Z_n^d$ with $|A|\geq n^d/2$ by looking at disjoint intervals of length $Cn^2\log(n/\epsilon)/\mu$. Since a lazy chain is delayed by a factor of~$2$, this finally shows that (for a different constant $C$)
\[
\max_{\eta_0}\cps{\eta=(\eta_t)_{t\geq 0}: \, \max_x\estart{\tau_A}{x,\eta}\geq \frac{Cn^2\log (n/\epsilon)}{\mu}}\leq \epsilon.
\]
Using~\eqref{eq:mixqu} for $\epsilon=1/4$ and performing independent experiments immediately gives that 
\[
\max_{x,\eta_0}\estart{\tau_A}{x,\eta_0}\leq \frac{Cn^2}{\mu}
\]
and this concludes the proof. 
\end{proof}

\section{Evolving Sets} \label{sec:EvolvingSets}

In this section, we derive the theory of evolving sets in a Markovian random environment 
in order to prove Theorem~\ref{thm:QuenchedMixingTimeGeneral}. 

We first recall the definition of evolving sets in the context of a finite state Markov chain;
see \cite{LevPerWil}. Given a Markov chain $p(x,y)$ with state space $\Omega$ and a stationary 
distribution $\pi$, the corresponding evolving-set process $\{S_n\}_{n\ge 0}$ is a Markov chain
on subsets of $\Omega$ whose
transitions are described as follows. Let $Q$ be defined as in (\ref{eq:DefnQ}) (with $\nu$ being
$\pi$) and let $U$ be
a uniform random variable on $[0,1]$. If $S\subseteq \Omega$ is the present state, we let the next 
state $\til{S}$ be defined by
$$
\til{S}:=\left\{y\in \Omega: \frac{Q(S,y)}{\pi(y)}\ge U\right\}.
$$
Note that $\Omega$ and $\emptyset$ are absorbing states and it is immediate to check that
\begin{equation}\label{eq:ES1d}
\prcond{y\in S_{k+1}}{S_k}{}=\frac{Q(S_k,y)}{\pi(y)}.
\end{equation}
Moreover, one can describe the evolving set process as that process on subsets which 
satisfies the ``one-dimensional marginal'' condition~\eqref{eq:ES1d} and where these 
different events, as we vary $y$, are maximally coupled.

For later use we also define now $$
\psi_{p}(S):=1-\E{\sqrt{\frac{\pi(\til{S})}{\pi(S)}}},
$$
where $\til{S}$ is the first step of the evolving set process started from $S$ and when the transition probability for the Markov chain is $p$ and the stationary distribution $\pi$.

We next define completely analogously the evolving set process in the context 
of a time inhomogeneous Markov chain. Consider a time inhomogeneous Markov chain with state space
$\S$ whose transition matrix for moving from time $k$ to time $k+1$ 
is given by $p_{k+1}(x,y)$ where we assume that the 
probability measure $\pi$ is a stationary distribution for each $p_k$. 
In this case, we say that~$\pi$ is a stationary distribution for the inhomogeneous Markov chain.
Let $Q_k$ be defined 
as in (\ref{eq:DefnQ}) but with respect to $p_k$ and $\pi$. We then obtain a time 
inhomogeneous Markov chain $S_0,S_1,\ldots$ on subsets of $\S$ generated by 
$$
{S}_{k+1}:=\left\{y\in \S: \frac{Q_{k+1}(S_k,y)}{\pi(y)}\ge U_{k+1}\right\}
$$
where $(U_i)_{i}$ are i.i.d.\ random variables uniform on $[0,1]$. We call this the evolving set process
with respect to $p_1,p_2,\ldots$ and $\pi$. 


We now need to consider the Doob transform of the evolving set process. If $P_\zeta$ is the transition probability for the evolving set process when the environment is $\zeta$, then we define the Doob transform via
\[
\widehat{P}_\zeta(S,S') = \frac{\pi(S')}{\pi(S)} P_\zeta(S,S').
\] 

We now let $\psi(\zeta,S):=\estart{\psi_{p_{\eta_1}}(S)}{\zeta}$. For $r\in [\pi_\star,\frac{1}{2}]$, we let
$$
\psi(r):=\inf\{\psi(\zeta,S):\zeta\in\calE, \pi(S)\le r\}
$$
and $\psi(r):=\psi(\frac{1}{2})$ for $r\ge \frac{1}{2}$.

In the following, we let 
\[ 
S^{\#}:= \left\{ 
\begin{array}{ccl}
S & \mbox{ if } \pi(S)\le \frac{1}{2} \\
S^c & \mbox{ otherwise }
\end{array} \right.
\]
and 
$$
Z_n:=\frac{\sqrt{\pi(S^{\#}_n)}}{\pi(S_n)}.
$$
 
\begin{lemma}\label{lem:zprocess}
Let $\epsilon>0$ and $x\in \S$. If 
\[
n\geq \int_{4\pi(x)}^{4/\epsilon} \frac{du}{u\psi(u)},
\]
then $\estarth{Z_n}{\{x\},\eta_0}\leq \sqrt{\epsilon}$ for all $\eta_0$.
\end{lemma}

\begin{proof}[\bf Proof]
We fix $x,\eta_0$ and 
to simplify notation we do not include them in the notation.
We now get that almost surely
\begin{align*}
\econdh{\frac{Z_{n+1}}{Z_n}}{S_n,\eta_0,\ldots,\eta_n} = \econd{\frac{\pi(S_{n+1})}{\pi(S_n)} \cdot \frac{Z_{n+1}}{Z_n}}{S_n,\eta_0,\ldots,\eta_n} = \econd{\sqrt{\frac{\pi(S_{n+1}^{\#})}{\pi(S_n^{\#})}}}{S_n,\eta_0,\ldots, \eta_n}.
\end{align*}
Suppose first that $\pi(S_n)\leq 1/2$. Then 
\begin{align*}
\econd{\sqrt{\frac{\pi(S_{n+1}^{\#})}{\pi(S_n^{\#})}}}{S_n,\eta_0,\ldots, \eta_n} \leq \econd{\sqrt{\frac{\pi(S_{n+1})}{\pi(S_n)}}}{S_n,\eta_0,\ldots, \eta_n}.
\end{align*}
The Markov property of the environment now gives
\begin{align}\label{eq:set}
\begin{split}
\econd{\sqrt{\frac{\pi(S_{n+1})}{\pi(S_n)}}}{S_n,\eta_0,\ldots, \eta_n} &= \sum_{\eta}R(\eta_n,\eta) (1-\psi_{p_\eta}(S_n))\\& = 1 - \psi(\eta_n,S_n)\leq 1-\psi(\pi(S_n)).
\end{split}
\end{align}
Suppose next that $\pi(S_n)>1/2$. Then 
\begin{align*}
\econd{\sqrt{\frac{\pi(S_{n+1}^{\#})}{\pi(S_n^{\#})}}}{S_n,\eta_0,\ldots, \eta_n} \leq \econd{\sqrt{\frac{\pi(S^c_{n+1})}{\pi(S^c_n)}}}{S_n,\eta_0,\ldots, \eta_n}
\end{align*}
and using the Markov property of the environment as before (as well as the fact that $(S_n^c)_n$ is also an evolving set process) we obtain
\begin{align}\label{eq:complement}
\econd{\sqrt{\frac{\pi(S^c_{n+1})}{\pi(S^c_n)}}}{S_n,\eta_0,\ldots, \eta_n} \leq 1-\psi(\pi(S^c_n)).
\end{align}
Since the function $\psi$ is non-increasing, it follows that if $\pi(S_n)>1/2$, then $\psi(\pi(S_n^c))\geq \psi(\pi(S_n))$, and hence from~\eqref{eq:set} and~\eqref{eq:complement} we get that in all cases
\begin{align*}
\econd{\sqrt{\frac{\pi(S_{n+1}^{\#})}{\pi(S_n^{\#})}}}{S_n,\eta_0,\ldots, \eta_n} \leq 1-\psi(\pi(S_n)) = 1-f_0(Z_n),
\end{align*}
where following~\cite{MorrisPeres} we set $f_0(z): = \psi(1/z^2)$ which is non-decreasing. (Note that $\pi(S_n)=Z_n^{-2}$ when $Z_n\geq \sqrt{2}$, i.e.\ when $\pi(S_n)\geq 1/2$ and $\psi(x)=\psi(1/2)$ for $x\geq 1/2$.) Therefore, we conclude
\[
\econdh{Z_{n+1}}{Z_n}\leq Z_n(1-f_0(Z_n))
\]
and hence using~\cite[Lemma~11 (iii)]{MorrisPeres} we get that $\estarth{Z_n}{}\leq \sqrt{\epsilon}$ for all 
\[
n\geq \int_{4\pi(x)}^{4/\epsilon} \frac{du}{u\psi(u)}
\]
and this finishes the proof.
\end{proof}

\begin{remark}\label{rem:newproof}
\rm{
We now explain the changes in the proof of the lemma above needed to justify Remark~\ref{rem:variant}. The matrix $R$ is replaced by $R(\zeta,\ell, \eta)$, where $\ell \in \{0,1\}$ depending on whether the walk made an actual step or not, i.e.\ $R(\zeta,0,\eta)=\1(\zeta=\eta)$ and $R(\zeta,1,\eta)=R(\zeta,\eta)$. In the definition of $\psi$, the matrix $p_{\eta}$ is replaced by $(p_{\eta}+I)/2$. The proof of the lemma above then remains unchanged with this new notation.
}	
\end{remark}

\begin{lemma}\label{lem:Lemma17.12}
If $(S_k)_{k\ge 0}$ is the evolving set process relative to an inhomogeneous Markov chain~$(X_k)$
and stationary distribution $\pi$, then
$$
\prstart{X_k=y}{x}=\frac{\pi(y)}{\pi(x)}\prstart{y\in S_k}{\{x\}}.
$$
\end{lemma}

\begin{proof}[\bf Proof]
In the case of a homogeneous Markov chain, this is Lemma 17.12 in \cite{LevPerWil}.
The proof for the inhomogeneous case goes through verbatim.
\end{proof}

\begin{lemma}\label{lem:Lemma17.13}
If $\{S_k\}_{k\ge 0}$ is the evolving set process relative to an inhomogeneous Markov chain
with stationary distribution $\pi$, then starting from any initial state,
$\{\pi(S_k)\}_{k\ge 0}$ is a martingale.
\end{lemma}

\begin{proof}[\bf Proof]
In the case of a homogeneous Markov chain, this is Lemma 17.13 in \cite{LevPerWil}.
The proof for the inhomogeneous case goes through verbatim.
\end{proof}

\begin{lemma}\label{lem:l2}
For all fixed environments $\eta=(\eta_i)_{i\geq 0}$ and all $x\in \S$ we have
\[
\chi\left(\prstart{Y_n=\cdot}{x,\eta}, \pi\right) \leq \estarth{Z_n}{x,\eta}.
\]
\end{lemma}

\begin{proof}[\bf Proof]

In the homogeneous case this is~\cite[equation (24)]{MorrisPeres}. The proof for the inhomogeneous case goes through verbatum using Lemmas~\ref{lem:Lemma17.12} and~\ref{lem:Lemma17.13}.
\end{proof}

\begin{proof}[\bf Proof of Theorem~\ref{thm:QuenchedMixingTimeGeneral}]

By Markov's inequality we obtain
\begin{align*}
\mathcal{P}_{\zeta}\left(\eta:\, \chi(\prstart{Y_n=\cdot}{x,\eta}, \pi)\geq \epsilon^{1/4}\right) \leq \epsilon^{-1/4} \mathbb{E}_{\zeta}\left[\chi(\prstart{Y_n=\cdot}{x,\eta}, \pi)\right],
\end{align*}
where the last expectation is taken over the environment $\eta$ started from $\zeta$. From Lemma~\ref{lem:l2} we can upper bound the right hand side by 
$\epsilon^{-1/4} \estarth{Z_n}{\{x\},\zeta}$. From Lemma~\ref{lem:zprocess} we get that $\estarth{Z_n}{\{x\},\zeta}\leq \sqrt{\epsilon}$ for all 
\[
n\geq 1+\int_{4\pi(x)}^{4/\epsilon} \frac{du}{u\psi(u)}.
\]
Lemma 10 in \cite{MorrisPeres} implies that for all $p$ and $S$,
$$
\psi_{p}(S)\geq \frac{\gamma^2}{2(1-\gamma)^2}\varphi^2_{p}(S).
$$
Therefore, taking expectations and using Jensen's inequality we get for all $\zeta$ and $S$
$$
\psi(\zeta,S)\geq \frac{\gamma^2}{2(1-\gamma)^2}\varphi^2(\zeta,S),
$$
and hence for all $r$
\[
\psi(r)\geq \frac{\gamma^2}{2(1-\gamma)^2}\phi^2(r).
\]
Thus this gives that for all 
\[
n\geq 1+ \frac{2(1-\gamma)^2}{\gamma^2}\int_{4\pi(x)}^{4/\epsilon} \frac{du}{u\phi^2(u)} 
\]
we have $\estarth{Z_n}{\{x\},\zeta}\leq \sqrt{\epsilon}$ and hence this completes the proof.
\end{proof}

\section*{Acknowledgements}

We thank Microsoft Research for its hospitality where parts of this work were completed. The third author also acknowledges the support of the Swedish Research Council and the Knut and Alice Wallenberg Foundation.

\bibliographystyle{plain}
\bibliography{biblio}

\end{document}